\documentclass[reqno, 11pt,a4size]{amsart}
\usepackage{graphicx, enumerate, amsmath, amssymb,amsthm,mathptmx}
\newtheorem{thm}{Theorem}[section]
\newtheorem{cor}[thm]{Corollary}
\newtheorem{lem}[thm]{Lemma}
\newtheorem{prop}[thm]{Proposition}
\newtheorem{result}[thm]{Result}
\theoremstyle{definition}

\theoremstyle{remark}

\numberwithin{equation}{section}
\newcommand{\norm}[1]{\left\Vert#1\right\Vert}
\newcommand{\abs}[1]{\left\vert#1\right\vert}

\newcommand{\D}{\mathbb{D}}

\newcommand{\cx}{{\mathbb{C}}}

\newcommand{\h}{\mathbb{H}}
\newcommand{\p}{\mathbb{P}}
\newcommand{\psm}{\sum_{\substack{\alpha_1+\alpha_2\leq k\\ \alpha_3+\alpha_4\leq k}}}
\newcommand{\psk}{\sum_{\abs{\alpha}\leq k}}

\newcommand{\dbar}{\overline{\partial}}
\newcommand{\dom}{\mathrm{Dom}}

\newcommand{\img}{{\mathrm{Img}}}
\newcommand{\tensor}{\otimes}
\newcommand{\csor}{\widehat{\otimes}}
\newcommand{\pss}{\widetilde{W}}

\newcommand{\wk}{W^k}

\begin{document}

\title{Sobolev Regularity  of  the $\overline{\partial}$-equation on the Hartogs Triangle}%
\author{Debraj Chakrabarti}%
\address{TIFR Centre for Applicable Mathematics, Sharada Nagar, Chikkabommasandra, Bengaluru- 560065, India.}
\email{debraj@math.tifrbng.res.in}
\author{Mei-Chi Shaw}
\address{Department of Mathematics, University of Notre Dame, Notre Dame, IN 46556, USA}%
\email{mei-chi.shaw.1@nd.edu}
\thanks{The
second-named author is partially supported by NSF grants.}

\dedicatory{Dedicated to the memory  of Prof.  Jianguo Cao}%
\begin{abstract}
The regularity of the  $\overline{\partial}$-problem  on the domain $\{\left\vert{z_1}\right\vert<\left\vert{z_2}\right\vert<1\}$ in 
$\mathbb{C}^2$ is studied using $L^2$-methods.
 Estimates are obtained for the canonical solution in weighted  $L^2$-Sobolev spaces with a weight that is singular at the 
point $(0,0)$.  
In particular,   the singularity of the Bergman projection for the Hartogs triangle  is contained at the singular point and it does not propagate. \end{abstract}
\maketitle
\section{Introduction}
The {\em Hartogs Triangle}, the bounded pseudoconvex domain $\h$ in $\cx^2$ given by 
$\h=\{(w_1,w_2)\in\cx^2\mid \abs{w_1}<\abs{w_2}<1\}$
is a venerable source of counterexamples to conjectures in complex analysis. 
 The boundary  of $\h$  has a serious singularity at $(0,0)$ near which it cannot be represented as a graph. Though a lot is 
known about $\h$, not all its
mysteries have been uncovered yet.  It is an important  yet simple model domain which needs to be understood thoroughly
in any program of extending classical results of several complex variables from smoothly bounded pseudoconvex
domains to more general domains.
 In this article, we consider the regularity of the $\dbar$-problem
on $\h$ in the $L^2$-Sobolev topology.

 Using integral representations,
the regularity of the $\dbar$-problem has been 
investigated  on $\h$ in \cite{cc, michel},  with estimates in the spaces $\mathcal{C}^{k,\alpha}$  
(functions and forms in $\mathcal{C}^k$,
whose $k$-th partial derivatives are   H\"{o}lder continuous of exponent $\alpha$.)
The remarkable outcome of these investigations is 
that for every  $\dbar$-closed  $(0,1)$-form $g$ on $\h$ of class $\mathcal{C}^{k,\alpha}$, there is a function $u$ on $\h$, 
also of class $\mathcal{C}^{k,\alpha}$ 
such that $\dbar u =g$, and this function $u$ is given by an explicit integral formula.
Note that the $\dbar$-problem is  {\em not}  globally regular on $\h$, i.e., there is a $\dbar$-closed $(0,1)$-form $h$ on $\h$, such that while
$h\in\mathcal{C}^\infty(\overline{\h})$,
 for every  $u$ satisfying $\dbar u= h$,  we have $u\not\in \mathcal{C}^\infty(\overline{\h})$ (see \cite{cc}.)  In contrast, when a domain
 is pseudoconvex with smooth boundary (see \cite{Kohn}),  or its closure has  a Stein neighborhood basis
 (see \cite{Dufresnoy}), one can solve the  $\dbar$-problem to obtain a solution smooth up to the boundary,  provided the data is smooth.

However, it is difficult to use the integral representation method to obtain information about regularity
in Sobolev spaces.   We use a method similar  in spirit to that used in \cite{michel} to obtain estimates in Sobolev spaces
for the canonical solution of the $\dbar$-equation in $\h$. We use the fact that $\h$ is biholomorphic
to a product domain $\p$ to transfer the problem from $\h$ to $\p$   (see Section~\ref{sec-prod} below.)
This opens up the possibility of using 
the technique of \cite{chak-shaw}. The fact that one of the factors in the product representation of $\h$ is non-Lipschitz
causes some technical problems in applying the results of \cite{chak-shaw} but these are easily overcome.
 This leads to estimates in Sobolev-type spaces with weights
singular at the bad point $(0,0)$.

The use of weights in the  $L^2$-method is of course classical. In the context of non-smooth domains, it seems that 
singular weights are a natural device to control the behavior of functions and forms near the singular part of the boundary.
Such weights also arise naturally in recent attempts to generalize classical estimates on the $\dbar$- and $\dbar$-Neumann
problems from smooth to non-smooth strictly pseudoconvex domains (see \cite{ehsani2, ehsani,ehsani1}.)

While the Hartogs triangle is rather special,   right now the method used here seems to be the only 
technique available to study the question treated in this paper. Of course, we can extend the method to related 
``Product-type" singularities. It will be very interesting to have a general technique to deal with the regularity 
 in Sobolev spaces of the $\dbar$-problem on singular  domains such as $\h$.

\subsection{Acknowledgements} The authors thank the referee for his  detailed
 comments and suggestions.  The first-named author thanks  Dr.~S.  Gorai for pointing out an error in 
 the first  version of this paper, and Prof.~M. Vanninathan for helpful hints on Sobolev spaces
on nonsmooth domains. He also thanks  Prof.~M. Ramaswamy, the Dean of the TIFR Centre for Applicable Mathematics,
for her active support of this research.

\section{Sobolev estimates }

 Let $\ell$ be an  integer, and let $L^2(\h,\ell\Phi)$ denote the space of  locally integrable functions $f$  on $\h$ for which the 
norm defined by 
\begin{equation}\label{eq-lphinorm}
\norm{f}_{L^2(\h,\ell\Phi)}^2=\int_\h \abs{w_2}^{2\ell} \abs{f(w)}^2 dV(w) 
\end{equation}
is finite, where $w=(w_1,w_2)$ are the standard coordinates on $\h$, and  here and in the sequel $dV$ denotes Lebesgue measure on Euclidean space.
$\Phi$ here denotes the harmonic function $\Phi(w)= -2 \log\abs{w_2}$ whose multiples are used as weights.
Other related notation  is explained  in Section~\ref{sec-prelim} below.
 Then $\ell=0$ corresponds to 
the usual unweighted  $L^2$-space on $\h$, positive values of $\ell$ correspond to allowing functions to blow up in a controlled way at $0$,
and negative values of $\ell$ correspond to forcing functions to vanish in a weak sense at the point $0$.
 We let $L^2_{0,1}(\h,\ell\Phi)$ denote the space of 
$(0,1)$-forms on $\h$ with coefficients in $L^2(\h,\ell\Phi)$.  On  a space of forms whose coefficients lie in a Hilbert space
 (e.g., $L^2_{0,1}(\h,\ell\Phi)$ here,
and the spaces $W^k_{0,1}(\h,\ell\Phi)$ and $W^k_{0,1}(\p,\ell\Phi)$  defined below),
according to standard convention, we impose a Hilbert space norm whose square is the sum of the squares of  the norms of the coefficients.
 It follows from H\"{o}rmander's theory of  $L^2$-estimates
for the $\dbar$-equation (see  Section~\ref{sec-prelim} below) that given a $\dbar$-closed $f$ in $L^2_{0,1}(\h,\ell\Phi)$, there is a $u$ 
 in $L^2(\h,\ell\Phi)$ such that $\dbar u=f$, and we have 
\begin{equation}\label{eq-l2onh}\norm{u}_{L^2(\h,\ell\Phi)}\leq \sqrt{e} \norm{f}_{L^2_{0,1}(\h,\ell\Phi)}.\end{equation}
(Where $e$ { is} the  base of natural logarithms.) By a standard weak compactness argument,
 among all such solutions $u$ there is a $u_\ell$ of smallest norm, which is 
the {\em (weighted) canonical} solution of $\dbar u=f$,
with weight $\ell\Phi$.  The aim  of this article  is to understand the regularity of $u_\ell$ in terms of that of $f$.

Let $k$ be a non-negative integer, and let $\ell$ be an integer. We introduce the weighted Sobolev space 
$W^k(\h,\ell\Phi)$ of locally integrable functions on $\h$ in the following way. 
For a multi-index  $\alpha=(\alpha_1,\alpha_2,\alpha_3,\alpha_4)$ of non-negative integers,
write $\abs{\alpha}= \sum_{j=1}^4 \alpha_j$, and 
let
\begin{equation}\label{eq-dalpha}  D^\alpha = \frac{\partial^{\abs{\alpha}}}
{\partial w_1^{\alpha_1}\partial{ \overline{w_1}}^{\alpha_{2} }
\partial w_2^{\alpha_3}\partial \overline{w_2}^{\alpha_{4} }},\end{equation}
and define  the  space $W^{k}(\h, \ell\Phi)$
 by the finiteness of the norm
\begin{equation}\label{eq-psskhdefn}
\norm{f}_{\wk(\h,\ell\Phi)}^2 = 
 \psk\int_{\h}\abs{w_2}^{2\ell}\abs{D^\alpha f(w)}^2 dV(w),
\end{equation}
where the derivatives are in the  weak  sense. We will refer to $\wk(\h,\ell\Phi)$ as the  weighted Sobolev space 
of order $k$ on $\h$ with weight $\ell\Phi$. We let $\wk_{0,1}(\h,\ell\Phi)$ be the space of 
 $(0,1)$-forms on $\h$ with coefficients in $W^k(\h,\ell\Phi)$.  The main result of this paper is:
\begin{thm} \label{thm-main} For every non-negative integer $k$ there is a  constant $C>0$, such that for each
$\dbar$-closed  $g$  in $W^{2k}_{0,1}(\h,\ell\Phi)$, 
 the canonical  solution $u_\ell$ of $\dbar u_\ell =g$ is in $\wk(\h, (\ell+2k)\Phi)$,  and satisfies an estimate 
\begin{equation}\label{eq-estimate}
\norm{u_\ell}_{\wk(\h,(\ell+2k)\Phi)} \leq C \norm{g}_{W^{2k}_{0,1}(\h,\ell\Phi)}
\end{equation}
\end{thm}

Note that the order of Sobolev differentiability of the solution  is half of that of the data,   and  the weight  factor in the norm changes
from $\abs{w_2}^{2\ell}$ to $\abs{w_2}^{2(\ell+2k)}$, indicating that the solution $u_\ell$  may have much more
rapid growth near 0 than  $g$ has.

No claim can be made for the optimality of the estimate given in  \eqref{eq-estimate}. Indeed,  the seeming loss of smoothness
from $W^{2k}$ to $W^k$  is illusory, arising from the use of the estimates given in \eqref{eq-pssinclusions} below,
and could in principle be avoided by introducing special weighted Sobolev spaces 
adapted to the Hartogs Triangle, but we have chosen to formulate the result in terms of the simpler spaces $W^k(\h,\ell \Phi)$.
We are interested in quantifying the possible blowup of the solution of the $\dbar$-equation on $\h$ with smooth data, and
this is deduced in Corollary~\ref{cor-blowup} below starting from Theorem~\ref{thm-main}.

The case $\ell=0$ corresponds to the usual canonical solution. In this case we can deduce the following corollary regarding the blowup
of the solution of the $\dbar$-equation at the point (0,0):
\begin{cor}\label{cor-blowup}
Let $g\in\mathcal{C}^\infty_{0,1}(\overline{\h})$ be a $\dbar$-closed $(0,1)$-form  smooth up to the boundary on  $\h$. Then the canonical
solution $u_0$ of the equation $\dbar u_0= g$ is smooth on $\h$ and extends smoothly  up to all points of the  boundary except possibly  at the point $(0,0)$.
If $\alpha$ is a multi-index and $D^\alpha$ is as  in \eqref{eq-dalpha}, we have
\begin{equation}\label{eq-blowup}
\int_\h \abs{w_2}^{4\abs{\alpha}} \abs{D^\alpha u_0(w)}^2 dV(w) <\infty.
\end{equation} 
\end{cor}

\begin{proof}
Since $g\in\mathcal{C}^\infty_{0,1}(\overline{\h})$, for every nonnegative
integer $k$, we have $g\in W^{2k}_{0,1}(\h)$, and consequently by Theorem~\ref{thm-main},
the solution $u_0$ is in $W^k(\h, 2k\Phi)$. If $\mathbb{B}$ be an open ball in $\cx^2$, such that $0\not\in \overline{\mathbb{B}}$,
the weight $\abs{w_2}^{4{k}}$ in   the definition of the Sobolev space  $W^k(\h, 2k\Phi)$ is smooth and  bounded away from zero and 
therefore  the restriction of functions in $W^k(\h,2k\Phi)$ to $\mathbb{B}\cap\h$
belong to $W^k({\mathbb{B}}\cap\h)$. Conversely, since $ {\mathbb{B}}\cap\h$ is Lipschitz,
by standard extension results, every  function in $W^k({\mathbb{B}}\cap\h)$ may be extended to a function in $W^k(\h)$.
Since this holds for each $k$, the restriction of $u_0$ to $\mathbb{B}\cap \h$ is in  $\mathcal{C}^\infty(\overline{\mathbb{B}\cap\h})$.
So the canonical solution
$u_0\in\mathcal{C}^\infty(\overline{\h}\setminus \{(0,0)\})$. 
The finiteness of \eqref{eq-blowup} now follows from 
\eqref{eq-estimate}.  
\end{proof}

Recall that the Bergman projection $B$ is defined as the orthogonal projection operator
from $L^2(\h)$ onto the closed subspace $L^2(\h)\cap \mathcal  O(\h)$, where $\mathcal O(\h)$ is the space of holomorphic functions
on $\h$.  We also have the following regularity and irregularity results for the Bergman projection on $\h$. 

\begin{thm}\label{thm-2} For $k\geq 0$, the Bergman projection $B$ maps the  Sobolev space $W^{2k+1}(\h)$ (without weight)  continuously 
into the  weighted holomorphic Sobolev space \[ W^{k}\left(\h,2k\Phi\right)\cap \mathcal{O}(\h).\] It follows that if
$f\in\mathcal{C}^\infty(\overline{\h})$, then $Bf \in \mathcal{C}^\infty(\overline{\h}\setminus\{0\})\cap\mathcal{O}(\h)$.
On the other hand,  $B$ does not map the space $\mathcal C^\infty_0(\h)$  of smooth functions compactly supported in $\h$ into 
 $W^{1}(\h)\cap\mathcal{O}(\h).$\end{thm}

Note that this result shows that the singularity of the Bergman projection for the Hartogs triangle  is contained at the singular point and it does not propagate.  
\section{H\"{o}rmander's existence theorem} \label{sec-prelim}
For a domain $\Omega$ in complex Euclidean space, and  a  real-valued continuous  function $\psi$ on $\Omega$,  recall that
 $L^2(\Omega,\psi)$ denotes  the space of   locally-integrable  functions $f$ on $\Omega$ for which the {\em weighted norm}
\[ \norm{f}_{L^2(\Omega,\psi)}^2 = \int_\Omega \abs{f}^2 e^{-\psi}dV \]
is finite. 
We denote by $L^2_{p,q}(\Omega,\psi)$  the space of $(p,q)$-forms with coefficients in the space $L^2(\Omega,\psi)$.
These are Hilbert spaces under the obvious inner products.

In this paper, on a domain in $\cx^2$,  we will use the  harmonic weight function $\Phi$, given by
\begin{equation}\label{eq-Phi} \Phi(w_1,w_2) = -2 \log\abs{w_2},\end{equation}
which is continuous provided the domain  does not intersect the complex line $\{w_2=0\}$.
Since $e^{-\ell\Phi(w)} = \abs{w_2}^{2\ell}$, this also explains the notations $L^2(\h,\ell\Phi)$ and 
$W^k(\h,\ell\Phi)$ adopted in the previous section for the spaces with norms \eqref{eq-lphinorm} and \eqref{eq-psskhdefn} 
respectively.

 The cornerstone of the $L^2$-theory of 
$\dbar$-operator is the following famous theorem of H\"{o}rmander (\cite[Theorem~2.2.1$^\prime$]{hor}, see also \cite{av} and 
the expositions in \cite{chen-shaw, straube}):
\begin{result}\label{res-hor} Let $\Omega\Subset\cx^n$ be pseudoconvex, and let $\psi\in \mathcal{C}^2(\Omega)$ be a 
strictly plurisubharmonic weight function on $\Omega$. For $z\in \Omega$, denote by $\mu(z)$ the smallest eigenvalue 
of the complex Hessian matrix $\left(\frac{\partial^2 \psi}{\partial z_j\partial \overline{z_k}}(z)\right)_{1\leq  j,k\leq n}$. 
If $\lambda = \inf_{z\in\Omega} \mu(z)>0$, then for  any  $\dbar$-closed
$g\in L^2_{p,q}(\Omega,\psi)$  $(p,q)$-form, $q>0$,  there is a $u\in L^2_{p,q-1}(\Omega,\psi)$ such that $\dbar u=g$,
satisfying the  estimate
\[ \norm{u}_{L^2_{p,q-1}(\Omega,\psi)}\leq \frac{1}{\lambda q} \norm{g}_{L^2_{p,q}(\Omega, \psi)}.\] 
\end{result}
From this the estimate \eqref{eq-l2onh} on $\h$ can be deduced as follows. 
We use the weight $\psi$ on $\h$, where  $\psi(w) = \frac{1}{2}\abs{w}^2+\ell\Phi(w)$. Then the space $L^2(\h,\psi)$
is the same as $L^2(\h,\ell\Phi)$, and the norms are equivalent. In fact it is easy to see that 
\begin{equation}\label{eq-equivalence}
 \frac{1}{\sqrt{e}}\norm{f}_{L^2(\h,\ell\Phi)}\leq \norm{f}_{L^2(\h,\psi)}\leq\norm{f}_{L^2(\h,\ell\Phi)}. 
\end{equation}
But $\psi$ is strictly plurisubharmonic,  and both  eigenvalues of its complex Hessian are identically  1, so in Result~\ref{res-hor}, $\lambda=1$. 
Let $g\in L^2_{0,1}(\h, \psi)$,  with $\dbar g=0$. 
Therefore, there is   a $u$ on $\h$ such that $\dbar u=g$ and $\norm{u}_{L^2(\h,\psi)}\leq \norm{g}_{L^2_{0,1}(\h,\psi)}$.
 Combining this with \eqref{eq-equivalence}, the estimate \eqref{eq-l2onh} follows.

\section{The product model of the Hartogs triangle}\label{sec-prod}
The Hartogs Triangle  is biholomorphic to the product domain $\p=\D\times \D^*$,
where $\D$ is the unit disc $\{z\in\cx\mid \abs{z}<1\}$ and $\D^*$ is the punctured unit disc $\{z\in\cx\mid 0<\abs{z}<1\}$.
The explicit map 
$ \mathsf{F}: \h \to \p $
is given by 
$ (w_1, w_2) \mapsto \left(\frac{w_1}{w_2}, w_2\right),$
and the inverse $\mathsf{G}= \mathsf{F}^{-1}: \p\to \h$ is given by
$  (z_1,z_2)\mapsto (z_1z_2, z_2).$   
This product representation allows us to study the regularity of the $\dbar$-equation on the Hartogs triangle 
using  the technique of \cite{chak-shaw}. Note that the biholomorphisms $\mathsf{F}$ and $\mathsf{G}$ are 
singular at the boundary, so we need to understand how spaces of functions and forms transform under these maps.

Given a  locally integrable function or form $f$ on $\h$,  we let $\mathsf{G}^* f$ denote the pullback of $f$ to $\p$. Similarly,  given
a locally integrable function or form   $g$ on $\p$, we denote by $\mathsf{F}^*g$ its pullback to  a form or function on $\h$.  We now 
consider the mapping properties of  the linear mappings  $\mathsf{F}^*$ and $\mathsf{G}^*$ on weighted  Sobolev spaces.
We denote by $\mathsf{F}^*_1$  and $\mathsf{F}^*_0$, the action of the operator $\mathsf{F}^*$ on $(0,1)$-forms and functions respectively,
and with a similar meaning for $\mathsf{G}^*_1$ and $\mathsf{G}^*_0$.

We define weighted Sobolev spaces $W^k(\p,\ell\Phi)$ on the domain $\p$ by the finiteness of the norm
\[\norm{g}_{\wk(\h,\ell\Phi)}^2 = 
 \psk\int_{\p}\abs{z_2}^{2\ell}\abs{D^\alpha g(z)}^2 dV(z).
\]


\begin{lem}\label{prop-fgmapping}  For each non-negative integer $k$,  and for $m\in \mathbb{Z}$, the pullback operator 
 $\mathsf{F}^*_0$ maps the space  $\wk(\p, (m+1)\Phi)$ continuously and injectively to $\wk(\h,(m+k)\Phi)$.
 Further,  for $k=0$, 
$\mathsf{F}^*_0$  is  actually an isometric isomorphism of the Hilbert space $L^2(\p,(m+1)\Phi)$ with the Hilbert space
$L^2(\h,m\Phi)$, and consequently the inverse mapping $\mathsf{G}^*_0$ is also an isometry from $L^2(\h,m\Phi)$ to 
$L^2(\p,(m+1)\Phi)$.

Also, for each non-negative integer $k$, and for  each  $\ell\in\mathbb{Z}$,
the map $\mathsf{G}^*_1$ maps  the  space of forms $\wk_{0,1}(\h,\ell\Phi)$ continuously and injectively
into $\wk_{0,1}(\p,(\ell+1)\Phi)$.
\end{lem}
We will allow ourselves, in this proof and the sequel, the standard abuse of notation by which $C$ stands for an arbitrary
constant, with possibly different values at different occurrences.
\begin{proof}
Since $\mathsf{G}$ and $\mathsf{F}$ are  biholomorphisms inverse to each other,  it follows that the operators
$\mathsf{G}^*$ and $\mathsf{F}^*$ are also inverses to each other. In particular,  they are both injective.

Let $f$  be a locally integrable function on $\p$ and let $g=\mathsf{F}^*_0 f$. Then $g(w_1,w_2)= f\left(\displaystyle{\frac{w_1}{w_2}}, w_2\right)$. 
Using the chain rule repeatedly (i.e., the  Fa\`{a} di Bruno formula, cf. \cite{fdb}) we see that there is an estimate of the form
\[ \abs{D^\alpha_w g(w_1,w_2)} 
\leq \frac{C}{\abs{w_2}^{\abs{\alpha}}}\sum_{\abs{\beta}\leq\abs{\alpha}} \abs{\left(D^\beta_z f\right)\left(\frac{w_1}{w_2}, w_2\right)}.\]

Now, we have
\begin{align*}\norm{\mathsf{F}^*_0f}_{\wk(\h, (m+k)\Phi)}^2 &= \psk \int_{\h} \abs{w_2}^{2(m+k)}
 \abs{D^\alpha g(w)}^2 dV(w)\\
&\leq  C\psk \int_\h \abs{w_2}^{2(m+k)}\left( \frac{1}{\abs{w_2}^{2k}}
\abs{(D^\alpha_zf)\left(\frac{w_1}{w_2}, w_2\right)}^2\right)dV(w)\\
&\leq C \psk \int_\p \abs{z_2}^{2m}\abs{D^\alpha f}^2\abs{z_2}^2 dV(z)\\ 
&= C\norm{f}^2_{\wk(\p, (m+1)\Phi)},
\end{align*}
where,  in the last but one line, $\abs{z_2}^2$ represents the Jacobian factor in the change of variables.
Considering the case $k=0$ separately, we have
\begin{align*}
\norm{\mathsf{F}^*_0f}_{L^2(\h,m\Phi)} &= \int_{\h} \abs{w_2}^{2m} \abs{g(w)}^2 dV(w)\\
&= \int_\p \abs{z_2}^{2m} \abs{f(z)}^2 \abs{z_2}^2 dV(z)\\
&= \norm{f}_{L^2(\p,(m+1)\Phi)},
\end{align*}
which proves that $\mathsf{F}^*_0$ is an isometry from ${L^2(\p,(m+1)\Phi)}$  onto $L^2(\h,m\Phi)$.

Now, let $g=g_1d\overline{w_1}+ g_2 d\overline{w_2}$  be a $(0,1)$-form on $\h$. The pullback $f= \mathsf{G}^*_1 g$ is 
then given by $f= f_1d\overline{z_1}+ f_2 d\overline{z_2}$, where,
\[ \begin{cases} f_1(z) =g_1(z_1z_2,z_2)\overline{z_2}\\
f_2(z) = g_1(z_1z_2,z_2)\overline{z_1}+ g_2(z_1z_2, z_2).
 \end{cases}\]
Using the Fa\`{a} di Bruno formula again,  we obtain for some constants depending on $\alpha$:
\[ \begin{cases} \abs{D^\alpha f_1}\leq C\sum_{\abs{\beta}\leq \abs{\alpha}}\abs{D^\beta g_1}\\
\abs{D^\alpha f_2} \leq C\sum_{\abs{\beta}\leq \abs{\alpha}}\left( \abs{D^\beta g_1}+ \abs{D^\beta g_2}\right).
\end{cases}\]
Therefore,
\begin{align*}
\norm{\mathsf{G}^*_1g}^2_{\wk(\p, (\ell+1)\Phi)} &= \psk \int_\p \abs{z_2}^{2(\ell+1)} \left(\abs{D^\alpha f_1}^2 + \abs{D^\alpha f_2}^2 \right) dV(z)\\
&\leq C \psk \int_\h \abs{w_2}^{2\ell} \left(\abs{D^\alpha g_1}^2 + \abs{D^\alpha g_2}^2\right) dV(w)\\
&=C\norm{g}^2_{\wk(\h, \ell\Phi)},
\end{align*}
where again we have used the change of variables formula.
\end{proof}

\section{Canonical Solutions  on  $\h$ and $\p$}
\subsection{The Canonical Solution operator}\label{sec-cansol} In this paper we are concerned with the situation 
in which we want to solve on $\h$ the $\dbar$-problem  for a $\dbar$-closed $(0,1)$-form $g$, i.e., find
a function $u$ such that $\dbar u=g$. In view of this,   we confine our discussions to the action of the $\dbar$-operator
on functions, noting here that many of these constructions apply to forms of arbitrary degree.

Let $\Omega$ be a domain and $\psi$ be a continuous weight function on $\Omega$.   As usual, we consider the maximal 
realization of $\dbar$, which is a closed densely defined unbounded operator from $L^2(\Omega,\psi)$  to $L^2_{0,1}(\Omega,\psi)$,
 whose domain $\dom(\dbar)$ consists of all $f\in L^2(\Omega,\psi)$ such that in the distributional sense  $\dbar f \in L^2_{0,1}(\Omega,\psi)$.
If the range $\img(\dbar)$ of the operator $\dbar$ is a closed subspace of $L^2_{0,1}(\Omega,\psi)$,
we can use general functional
analytic methods to define a {\em bounded} solution operator $K:\img(\dbar)\to L^2(\Omega,\psi)$,  which maps a $g\in \img(\dbar)\subset 
L^2_{0,1}(\Omega,\psi)$ to the solution of smallest norm of the equation $\dbar u=g$ (equivalently,  we can say that $Kg$ is the unique solution 
of $\dbar u=g$ which is orthogonal to the {\em Bergman Space} $\mathcal{O}(\Omega)\cap L^2(\Omega,\psi)$.) We can extend $K$ to the whole 
of $L^2(\Omega,\psi)$ by declaring to be zero on $(\dom(\dbar))^\perp \subset L^2(\Omega,\psi)$. This $K$ is referred to the {\em canonical (or Kohn)
solution operator} of the $\dbar$-problem on $\Omega$ with weight $\psi$. In the theory of the $\dbar$-Neumann problem, we can
represent $K$ as $\dbar^*_\psi \mathsf{N}_{\psi, (0,1)}$, where  $\dbar^*_\psi$ is the Hilbert space adjoint of  the $\dbar$ operator, and
$\mathsf{N}_{\psi,(0,1)}$ is the $\dbar$-Neumann operator on the domain $\Omega$ with weight $\psi$ acting on $(0,1)$-forms.
The study of the regularity properties of $\mathsf{N}_{\psi,(0,1)}$ provides a powerful approach to the study of regularity of $K$ itself. Unfortunately
this  method is not available on the non-smooth domain $\h$ we are considering.

For technical reasons we would sometimes like to think of the canonical solution operator  as defined on the  orthogonal direct sum
  $L^2(\Omega,\psi)\oplus L^2_{0,1}(\Omega,\psi)$  and taking values in $L^2(\Omega,\psi)$.
This is achieved by  declaring the operator to be 0 on the functions
in $L^2(\Omega,\psi)$ and  extending linearly.

\subsection{$K^\ell_\h$ and $K^\ell_\p$.} 
 From the discussion in  Section~\ref{sec-prelim} it follows that there exists a canonical solution operator on the domain $\h$ 
with weight $\ell\Phi$ for each $\ell\in\mathbb{Z}$, where $\Phi$ is the harmonic function defined in \eqref{eq-Phi}.
 We denote this operator by $K^\ell_\h$. Then $K^\ell_\h$ is a bounded operator
from $L^2_{0,1}(\h,\ell\Phi)$ to $L^2(\h, \ell\Phi)$.

Similarly, there is for each $\ell\in\mathbb{Z}$, a canonical solution operator $K^\ell_\p$ for the $\dbar$-operator on $\p$.
Applying  Result~\ref{res-hor} to $\p$, with weight $\psi=\frac{1}{2}\abs{z}^2+ \ell\Phi$ gives us a solution to $\dbar v=g$ for
$g\in L^2_{0,1}(\p,\psi)\cap \ker(\dbar)$, satisfying the estimate $\norm{v}_{L^2(\p,\psi)}\leq \norm{g}_{L^2_{0,1}(\p,\psi)}$.
But $L^2(\p,\psi)$ and $L^2(\p,\ell\Phi)$ are the same space with equivalent norms, so we can solve $\dbar v=g$ with 
$g\in L^2_{0,1}(\p,\ell\Phi)\cap \ker(\dbar)$, where $v$  satisfies the estimate $\norm{v}_{L^2(\p,\ell\Phi)}\leq C\norm{g}_{L^2_{0,1}(\p,\ell\Phi)}$.
Now the existence of $K^\ell_\p$ follows as in Section~\ref{sec-prelim}.

We note  the relation between the canonical operators on $\h$ and $\p$:
\begin{lem}We have
\begin{equation}
\label{eq-canrel}
K^\ell_\h=\mathsf{F}_0^*\circ K^{\ell+1}_\p\circ \mathsf{G}_1^*.\end{equation}
\end{lem}
 \begin{proof}
Denote the operator defined by the right hand side of \eqref{eq-canrel} by $S_\ell$. This  maps $(0,1)$-forms on $\h$  to functions 
on $\h$ and satisfies
\begin{equation}\label{eq-sell} K^{\ell+1}_{\p}\circ \mathsf{G}^*_1 =\mathsf{G}^*_0\circ S_\ell.\end{equation}
Since  the $\dbar$ operator commutes with pullbacks by holomorphic mappings, it follows that 
$S_\ell$ is a solution operator for $\dbar$,  i.e., $\dbar\left(S_\ell g\right) =g$, if $\dbar g=0$ on $\h$.
Further,  $S_\ell$ is bounded from $L^2_{0,1}(\h,\ell\Phi)$ to $L^2(\h,\ell\Phi)$, since we know 
from Lemma~\ref{prop-fgmapping} above that $\mathsf{G}^*_1$ is continuous from the
space $L^2_{0,1}(\h, \ell\Phi)$ to $L^2_{0,1}(\p,(\ell+1)\Phi)$
and that $\mathsf{F}^*_0$ is continuous from $L^2(\p,(\ell+1)\Phi)$ to $L^2(\p,\ell\Phi)$,
and by definition $K^{\ell+1}_\p$ is  continuous 
from $L^2_{0,1}(\p,(\ell+1)\Phi)$ to $L^2(\p,(\ell+1)\Phi)$. 

Suppose now  that $S_\ell\not= K^\ell_\h$ . Then there is a $g\in L^2_{0,1}(\h,\ell\Phi)$  with $\dbar g=0$, and a $u\in L^2(\h,\ell\Phi)$
such that $\dbar u=g$ and $\norm{u}_{L^2(\h,\ell\Phi)}< \norm{S_\ell g}_{L^2(\h.\ell\Phi)}$.  Note that $\mathsf{G}^*_1 g$ is $\dbar$-closed,
and consider the $\dbar$-problem on $\p$ given by $\dbar v =\mathsf{G}^*_1 g$.  Both $\mathsf{G}^*_0 u$ and $K^{\ell+1}_\p (\mathsf{G}^*_1g)$ are solutions of this
equation in $L^2(\p,(\ell+1)\Phi)$, and since $K^{\ell+1}_\p (\mathsf{G}^*_1g)$ is the canonical solution, we have
\begin{equation}\label{eq-cont1} 
\norm{K^{\ell+1}_\p(\mathsf{G}^*_1 g)}_{L^2(\p,(\ell+1)\Phi)}\leq \norm{\mathsf{G}^*_0 u}_{L^2(\p,(\ell+1)\Phi)}.
\end{equation}
Since $\mathsf{G}^*_0$ is an isometry by
Lemma~\ref{prop-fgmapping}, it follows that 
\begin{align*}
\norm{\mathsf{G}^*_0u}_{L^2(\p,(\ell+1)\Phi)}&= \norm{u}_{L^2(\h,\ell\Phi)}\\
&< \norm{S_\ell g}_{L^2(\h.\ell\Phi)}\\
&= \norm{\mathsf{G}^*_0S_\ell g}_{L^2(\p,(\ell+1)\Phi)}\\
&= \norm{K^{\ell+1}_\p(\mathsf{G}^*_1 g)}_{L^2(\p,(\ell+1)\Phi)},
\end{align*}
where we have used \eqref{eq-sell} in the last line. But this 
contradicts \eqref{eq-cont1} and we conclude therefore that $S_\ell = K^\ell_\h$.  
\end{proof}

\subsection{Representation of the Canonical Solution on the product domain $\p$}
In order to estimate the operator $K^\ell_\p$ on the product domain $\p=\D\times \D^*$, we want 
to use  \cite[Theorem~4.7]{chak-shaw} to represent it using terms of  the canonical solution operators  and Bergman 
projections of the factors $\D$ and $\D^*$.  This theorem is as follows  ($\csor$ denotes the Hilbert tensor product of
Hilbert spaces, i.e., the completion of the algebraic tensor product under its  natural hermitian  inner product, see \cite{chak-shaw}):
\begin{result}\label{thm-4pt7a}
  Let $\Omega_1\Subset\cx^{n_1}$ and $\Omega_2\Subset\cx^{n_2}$ be bounded {\em Lipschitz} domains, and let $\psi_1, \psi_2$ be 
continuous functions on $\Omega_1, \Omega_2$ respectively.  Suppose that, for $j=1,2$, the  $\dbar$-operator has closed range 
as an operator from $L^2(\Omega_j,\psi_j)$ to $L^2_{0,1}(\Omega_j,\psi_j)$. Then the $\dbar$-operator has closed range 
from $L^2(\Omega, \psi)$ to $L^2_{0,1}(\Omega,\psi)$, where $\Omega=\Omega_1\times\Omega_2\Subset \cx^{n_1+n_2}$, 
and $\psi=\psi_1+\psi_2$. Further
the canonical solution operator $K:L^2_{0,1}(\Omega,\psi)\to L^2(\Omega,\psi)$ 
restricted to the space of $\dbar$-closed $(0,1)$-forms has the representation
\begin{equation}\label{eq-k}  K = K_1\csor I_2 + \sigma_1 P_1\csor K_2,\end{equation}
where  $K_1,K_2$ are the canonical solution operators on $\Omega_1, \Omega_2$ respectively, $P_1$ is the harmonic 
projection on $\Omega_1$ and $\sigma_1$ is a linear operator which restricts to multiplication by $(-1)^d$ on the space 
of forms of total degree $d$ on $\Omega_1$
\end{result}
Of course, there is a second representation analogous to \eqref{eq-k} obtained by switching the roles of $\Omega_1$ and $\Omega_2$.

Unfortunately, one of the factors $\D^*$ of $\p$ is {\em not} Lipschitz, so  Result~\ref{thm-4pt7a} does not apply as 
stated  in the situation we are interested. 
However, we contend that the conclusion of Result~\ref{thm-4pt7a} still holds for $\Omega_1=\D$ and $\Omega_2=\D^*$ with weights
$\psi_1\equiv 0$ and $\psi_2 = \ell\phi$, where $\phi$ is the harmonic function on $\D^*$ given by 
\begin{equation}\label{eq-phi}\phi(z)=-2\log\abs{z}.\end{equation}

We first state a general result which we can apply to $\p$.  Let $\mathsf{H}_1$ and $\mathsf{H}_2$ be Hilbert spaces, and
let $T:\mathsf{H}_1\to\mathsf{H}_2$ be a densely defined  closed  linear operator from  a subspace $\dom(T)\subset\mathsf{H}_1$
to  a subspace $\mathsf{H}_2$. The {\em graph norm} $\norm{u}_{\Gamma(T)}$ of an element $u\in \dom(T)$ is defined by
$ \norm{u}_{\Gamma(T)}^2= \norm{u}_{H_1}^2 + \norm{Tu}_{H_2}^2$,
and since $T$ is closed, $\dom(T)$ is a Hilbert space in this norm.  Recall that a {\em core} of a densely defined operator $T$ is a 
subspace $\mathcal{G}\subset\dom(T)$ which is dense in $\dom(T)$ in the graph norm (cf. \cite[p. 155]{kr}.) After these definitions, we can state 
the slightly more general form of Result~\ref{thm-4pt7a}:
\begin{prop} \label{prop-g1g2} The hypotheses are the same as in Result~\ref{thm-4pt7a}, except that $\Omega_1$ and $\Omega_2$ are 
not assumed to be Lipschitz. Instead  we assume that there exists a core $\mathcal{G}_1$ of $\dbar: L^2(\Omega_1,\psi_1)\to
L^2_{0,1}(\Omega_1,\psi_1)$ and a core $\mathcal{G}_2$ of $\dbar: L^2(\Omega_2,\psi_2)\to
L^2_{0,1}(\Omega_2,\psi_2)$, such that the algebraic tensor product $\mathcal{G}_1\tensor \mathcal{G}_2$ 
is a core of  the operator $\dbar: L^2(\Omega,\psi)\to
L^2_{0,1}(\Omega,\psi)$,  then the same conclusion (in particular the representation \eqref{eq-k} of the canonical solution) holds.
\end{prop}

\begin{proof} We  only indicate the changes that need to be made in the proof of  Result~\ref{thm-4pt7a} as 
given in \cite{chak-shaw} in order to verify this more general statement. Note that the only way the Lipschitz 
condition is used in the proof of Result~\ref{thm-4pt7a} is to provide the cores $\mathcal{C}^\infty(\overline{\Omega_1}),
 \mathcal{C}^\infty(\overline{\Omega_2}), \mathcal{C}^\infty(\overline{\Omega})$ on $\Omega_1$, $\Omega_2$ and $\Omega$,
and to make sure that $\mathcal{C}^\infty(\overline{\Omega_1})\tensor
 \mathcal{C}^\infty(\overline{\Omega_2})$ is dense in the graph-norm $\Gamma(\dbar)$ in $\mathcal{C}^\infty(\overline{\Omega})$
and therefore in $\dom(\dbar)$. It is easy to check all the arguments in \cite{chak-shaw} continue to hold if we replace $\mathcal{C}^\infty(\overline{\Omega_1})$ by $\mathcal{G}_1$ and 
$ \mathcal{C}^\infty(\overline{\Omega_2})$ by $\mathcal{G}_2$.
\end{proof}

We now proceed to apply Proposition~\ref{prop-g1g2} to $\p=\D\times \D^*$. We  take $\Omega_1$ to be $\D$
and $\psi_1\equiv 0$. Then $K_1=K_{\D}$, the canonical solution operator on the unit disc $\D$ without any weight,
and $P_1=P_{\D}$, the Bergman projection on $\D$ in degree 0, and the zero operator in other degrees (since 
harmonic spaces vanish in other degrees.) The closed range property for $\dbar$ and the 
existence of the canonical solution operator is immediate from Result~\ref{res-hor} by using the weight $\psi=\frac{1}{2}\abs{z}^2$.

For $\Omega_2$,  we take the punctured disc $\D^*$. Let $\phi$ be  as in \eqref{eq-phi}. 
We take the weight $\psi_2$ to be $\ell\phi$.  Note that then  $L^2(\D^*,\ell\phi)$ has the norm
\[ \norm{f}_{L^2(\D^*,\ell\phi)}= \int_{\D^*}\abs{z}^{2\ell} f(z)dV(z).\]

We need to show that $\dbar: L^2(\D^*,\ell\phi)\to L^2_{0,1}(\D^*,\ell\phi)$ has closed range. For this we use the
same method as used in the proof of \eqref{eq-l2onh}. In Result~\ref{res-hor}, we let the weight $\psi$ to be
$\psi= \frac{1}{2}\abs{z}^2+\ell\phi$. This immediately shows that for any $gd\overline{z}\in L^2_{0,1}(\D^*,\ell\phi)$, there 
is a $v\in L^2(\D^*, \ell\phi)$ such that  $\dbar v =gd\overline{z} $, i.e. 
$ \frac{\partial v}{\partial \overline{z}}=g$,
and $v$ satisfies the estimate 
\[ \norm{v}_{L^2(\D^*,\ell\phi)} \leq \sqrt{e} \norm{g}_{L^2(\D^*,\ell\phi)}.\]
The existence of the canonical solution  follows as usual. We denote the canonical solution operator
by $K^\ell_{\D^*}$. It is a bounded operator from $L^2_{0,1}(\D^*,\ell\phi)$ to $L^2(\D^*,\ell\phi)$.

In order to apply Proposition~\ref{prop-g1g2} we also need cores $\mathcal{G}_1$ of  $\dbar: L^2(\D)\to L^2_{0,1}(\D)$,
and $\mathcal{G}_2$ of $\dbar: L^2(\D^*,\ell\phi)\to  L^2_{0,1}(\D^*,\ell\phi)$ such that $\mathcal{G}_1\tensor\mathcal{G}_2$
is a core of  $\dbar: L^2(\p,\ell\Phi)\to L^2_{0,1}(\p,\ell\Phi)$. We take $\mathcal{G}_1=\mathcal{C}^\infty(\overline{\D})$. Since
$\D$   has smooth boundary,  it follows that $\mathcal{G}_1$ is a core for the $\dbar$ operator on $L^2(\D)$.  Let 
$\mathcal{G}_2$ be the space of functions on $\D^*$ of the form $z^{-\ell} f$, where $f\in\mathcal{C}^\infty(\overline{\D})$.
We have the following:
\begin{lem}(1) $\mathcal{G}_2$ is a core of $\dbar: L^2(\D^*,\ell\phi)\to  L^2_{0,1}(\D^*,\ell\phi)$.

(2) $\mathcal{G}_1\tensor\mathcal{G}_2$
is a core of  $\dbar: L^2(\p,\ell\Phi)\to L^2_{0,1}(\p,\ell\Phi)$.
\end{lem}
\begin{proof} Let $g\in\dom(\dbar)\subset L^2(\D^*,\ell\phi)$. It follows that 
$z^\ell g\in L^2(\D)$ and $\dbar(z^\ell g)= z^\ell \dbar g \in L^2(\D)$. Therefore, $z^\ell g$ 
belongs to the domain of $\dbar$ as an operator on $L^2(\D)$. We take a sequence $\{f_\nu\}$ of 
forms in $\mathcal{C}^\infty(\overline{\D})$ converging in the graph norm of $\dbar$ on $L^2(\D)$ 
to $z^\ell g$. It is easy to see that $z^{-\ell}f_\nu$ converges to $g$ in the graph norm of $\dbar$ on 
$L^2(\D^*,\ell\phi)$. Part (1) follows.

Let $\mathcal{H}$ denote the forms on $\p=\D\times \D^*$ which are of the type $z_2^{-\ell}f$, where
$f\in\mathcal{C}^\infty(\overline{\D}\times\overline{\D})$. An argument analogous to the one in Part (1)
above shows that $\mathcal{H}$ is a core of the $\dbar$ operator acting on $L^2(\p,\ell\Phi)$.  Given
any $z_2^{-\ell}f\in \mathcal{H}$ we can approximate $f$ in the $\mathcal{C}^1$ norm on $\overline{\D}$ 
by elements of the algebraic tensor product $\mathcal{C}^\infty(\overline{\D})\tensor\mathcal{C}^\infty(\overline{\D})$
(cf. \cite[page 369]{hth}.)
From this the statement (2) follows immediately.
\end{proof}

Therefore, we obtain the following representation of the canonical solution $K^\ell_{\p}$ in terms 
of the factor domains $\D$ and $\D^*$. Note that in the second term of \eqref{eq-k}, the only term that 
is non-zero is the term corresponding to functions on $D$, since the harmonic projection vanishes in every 
other degree, and for this remaining term we have $\sigma_1=1$:
\begin{prop}\label{prop-klp}On the $\dbar$-closed $(0,1)$-forms in $L^2_{0,1}(\p,\ell\Phi)$, we have
\[ K^\ell_{\p} = K_\D\csor I_{\D^*} + P_\D \csor K^\ell_{\D^*},\]
where
\begin{itemize}\item $K_\D:L^2_{0,1}(\D)\to L^2(\D)$ is the canonical solution of the $\dbar$ equation on the
unit disc.  (Recall that by convention, we assume that canonical solution operators vanish on functions.)

\item $I_{\D^*}$ is the identity map on functions and forms on $\D^*$.

\item $P_\D:L^2(\D)\to L^2(\D)\cap \mathcal{O}(\D)$ is the Bergman projection. It is extended to 
$L^2_{0,1}(\D)$  by setting it equal to 0.

\item $K^\ell_{\D^*}:L^2_{0,1}(\D^*,\ell\phi)\to L^2(\D^*,\ell\phi)$ is the canonical solution operator on $\D^*$ with weight $\ell\phi$.
\end{itemize}

\end{prop}

\section{Proof of Theorem~\ref{thm-main}}

\subsection{Expression of $K^\ell_\h$ in terms of components}
Combining Proposition~\ref{prop-klp} with Lemma~\ref{prop-fgmapping} we obtain the following:
\begin{cor} \label{cor-rep}We have
\begin{equation}\label{eq-rep-klh}K^\ell_{\h} = \mathsf{F}_0^*\circ \left( K_{\D}\csor I_{\D^*}+ P_{\D}\csor K_{\D^*}^{\ell+1}\right)\circ \mathsf{G}_1^*.\end{equation}
\end{cor}

We need to estimate the various operators appearing  in \eqref{cor-rep} in Sobolev spaces in order to prove
Theorem~\ref{thm-main}. The regularity of $\mathsf{G}_1^*$ and $\mathsf{F}_0^*$ in partial Sobolev spaces
has already been discussed in Lemma~\ref{prop-fgmapping}.
We consider the Sobolev Space $W^k(\D)$ on the unit disc of order $k\geq 0$, which is given by 
the finiteness of the norm
\[ \norm{f}_{W^k(\D)}^2= \sum_{\alpha+\beta\leq k}\int_\D 
\abs{\frac{\partial^{\alpha+\beta}f }{\partial z^\alpha\partial\overline{z}^\beta }}^2 dV\]
For the disc  $\D$,   it is well-known from  potential theory that   $K_\D$ maps $\dbar$-closed forms in $W^k_{0,1}({\D})$ to 
functions in $W^{k+1}(\D)$, and the Harmonic projection, which is non-zero only in degree 0, is identical to the Bergman projection
which maps functions in $W^k(\D)$ to holomorphic functions in $W^k(\D)$ (condition~ ``R".)


\subsection{ Regularity of $K^\ell_{\D^*}$}We use Sobolev spaces with the weight $\phi$ as in \eqref{eq-phi}.
The norm in such a space $W^k(\D^*, \ell\phi)$ is  given by
\[ \norm{f}_{W^k(\D^*,\ell\phi)}^2= \sum_{\alpha+\beta\leq k}\int_\D \abs{z}^{2\ell}
\abs{\frac{\partial^{\alpha+\beta}f }{\partial z^\alpha\partial\overline{z}^\beta }}^2 dV.\]
With respect to these spaces, we have the following:
\begin{prop}\label{prop-dstarreg} For every nonnegative integer $k$, 
the operator $K^\ell_{\D^*}$ is bounded from  the Sobolev space $W^k_{0,1}(\D^*,\ell\phi)$  to the Sobolev space
 $W^k(\D^*, (\ell +k)\phi)$.
\end{prop}

\begin{proof} Let $g\in W^k(\D^*,\ell\phi)$ , and set $\widetilde{g}= z^{k+\ell}g$.  We claim that $\widetilde{g}\in W^k(\D)$, the unweighted
standard Sobolev space of order $k$ on the disc. Indeed, 
for $\alpha+\beta \leq k$, we have using the Leibniz rule:
\[ \frac{\partial^{\alpha+\beta} \widetilde{g}}{\partial z^\alpha \partial \overline{z}^\beta}=\sum_{j=0}^\alpha \binom{\alpha}{j}z^{k+\ell-j}
\frac{\partial^{\alpha+\beta-j}g}{\partial z^{\alpha-j}\partial\overline{z}^\beta}.\]
Note that each term in the sum on the right is in $L^2(\D)$, since by hypothesis $g\in W^k(\D^*,\ell\phi)$.  Further, using the fact that
$\abs{z^{k+\ell-j}}\leq \abs{z}^\ell$,  it easily follows that there is an estimate 
\begin{equation}\label{eq-a}
 \norm{\widetilde{g}}_{W^k(\D)}\leq C\norm{g}_{W^k(\D^*,\ell \phi)}.
\end{equation}
Let $\tilde{u}$ denote that canonical solution of the equation  $\dbar\tilde{ u} = \tilde{g}d\overline{z}$ in $L^2(\D)$ 
(without any weights.) Then we know that $\tilde{u}\in W^{k+1}(\D)$, and we have an estimate
\begin{equation}\label{eq-b} \norm{\tilde{u}}_{W^{k+1}(\D)}\leq C\norm{\tilde{g}}_{W^k(\D)}.\end{equation}
Set $u = z^{-(k+\ell)}\tilde{u}$. Then, on $\D^*$, we have $\dbar u = gd\overline{z}$  and 
\begin{align*}
\norm{\widetilde{u}}_{W^{k+1}(\D)}^2&=\norm{z^{k+\ell}{u}}_{W^{k+1}(\D)}^2\\
&= \sum_{\alpha+\beta\leq k+1}\int_\D\abs{\frac{\partial^{\alpha+\beta}}{\partial z^\alpha \partial \overline{z}^\beta}\left( z^{(k+\ell)}u\right)}^2dV\\
&\geq C \sum_{\alpha+\beta\leq k+1}\int_\D\abs{z}^{2(k+\ell)}\abs{\frac{\partial^{\alpha+\beta}u}{\partial z^\alpha \partial \overline{z}^\beta}}^2dV\\
&\geq C \norm{u}_{W^{k+1}(\D^*, (\ell+k)\phi)}.
\end{align*}
Combining this with \eqref{eq-a} and \eqref{eq-b}, we see that there is a linear solution operator  $g\mapsto u$, 
for $\frac{\partial}{\partial\overline{z}}$ 
on the punctured disc $\D^*$ which is continuous from $W^k(\D^*,\ell\phi)$ to $W^{k+1}(\D^*, (\ell+k)\phi)$.

Denote by $v=K_{\D^*}^\ell (gd\overline{z})$ the canonical solution of $\dbar v =gd\overline{z}$ in the weighted space $L^2(\D^*,\ell\phi)$.  Then
$v$ is of the form $u+h$, where  $h$ is a function in the Bergman space
$\mathcal{O}(\D^*)\cap L^2(\D^*,\ell\phi)$, and $u$ is the solution of $\dbar u= gd\overline{z}$ found above.  But then $h$ must be of the form $h=z^{-\ell}f$, 
where $f\in \mathcal{O}(\D)\cap L^2(\D)$.  Denote by $\D^*_{\frac{1}{2}} $ the punctured disc   $\{0<\abs{z}<\frac{1}{2}\}$ of radius $\frac{1}{2}$.
 A direct computation shows that  $h\in W^k\left(\D^*_{\frac{1}{2}}, (\ell+k)\phi\right)$.  Since by the last paragraph, $u\in W^{k+1}(\D^*, (\ell+k)\phi)$, it now follows that
$v\in W^k\left(\D^*_{\frac{1}{2}}, (\ell+k)\phi\right)$.

Now let $\chi$ be a cutoff on $\D$ which is identically 1 on $\{\abs{z}>\frac{1}{2}\}$, and vanishes in a neighborhood of 0.
By standard localization results, $\chi v \in W^{k+1}(\D)$. Combining with the fact that $v\in W^k\left(\D^*_{\frac{1}{2}}, (\ell+k)\phi\right)$,
it follows that $v\in W^k(\D^*, (\ell+k)\phi)$, and the result is proved.
\end{proof}

\subsection{Estimates on $K^\ell_{\h}$}  As in \cite{chak-shaw}, for an integer $k\geq 0$,
 we introduce the {\em weighted partial Sobolev space} $\pss^k(\p, \ell\Phi)$
by the finiteness of the norm:
\begin{equation}\label{eq-psskhdefn2}
\norm{f}_{\pss^k(\p,\ell\Phi)}^2 = 
 \psm\int_{\p}\abs{z_2}^{2\ell}\abs{D^\alpha f(z)}^2 dV(z),
\end{equation}
where $D^\alpha$ is as in \eqref{eq-dalpha}, the derivatives are in the  weak  sense, and  note the special  range of summation.  It is clear that
\begin{equation}\label{eq-pssinclusions} W^{2k}(\p,\ell\Phi)\subsetneq\pss^k(\p,\ell\Phi)\subsetneq W^k(\p,\ell\Phi),\end{equation}
with continuous inclusions.

We begin with the following lemma  which holds for every non-negative $k$  :
\begin{lem}\label{lem-psstensor}

$\pss^k(\p,\ell\Phi)= W^{k}(\D)\csor W^{k}(\D^*,\ell\phi)$.
\end{lem}
\begin{proof}
Were the domains $\D^*$ and $\D$ both Lipschitz  we could use the method of  \cite[Lemma~5.1]{chak-shaw} directly.  Since $\D^*$ is 
non-Lipschitz, we proceed as follows. For a multi-index $\nu\in \mathbb{Z}^4$, denote by ${\bf z}^\nu$ the 
Laurent-type monomial $z_1^{\nu_1}\overline{z_1}^{\nu_2}z_2^{\nu_3}\overline{z_2}^{\nu_4}$.   For $m\in \mathbb{Z}$, if $S(m)$ denotes the set
of monomials
$\{{\bf z}^{\nu}\mid \nu_1\geq 0, \nu_2\geq 0, \nu_3+\nu_4\geq m\}$, it is easy to see that the elements of $S(-\ell)$ form a
complete set in $L^2(\p,\ell\Phi)$, i.e. the linear span of $S(-\ell)$ is dense in $L^2(\p,\ell\Phi)$.  If $\alpha$ and   $D^\alpha$ 
are as in \eqref{eq-dalpha} with 
$\alpha_3+\alpha_4\leq k$, then  for ${\bf z}^\nu\in S(m)$,
the derivative 
$D^\alpha {\bf z}^{\nu}$ is a scalar multiple of an element of $S(m-k)$.  Further, every element of $S(m-k)$ arises (up to a multiplicative
factor) as a partial derivative of this sort.  By definition, a function
 $f\in \pss^k(\p, \ell\Phi)$, if and only if  $D^\alpha f\in L^2(\p,\ell\Phi)$,  for $\alpha$ with 
$\alpha_1+\alpha_2\leq k$ and 
$\alpha_3+\alpha_4\leq k$. It now follows easily
that $S(-\ell+k)$  is complete in $\pss^k(\p,\ell\Phi)$.
But an element of $S(-\ell+k)$  may be written as 
$(z_1^{\nu_1}\overline{z_1}^{\nu_2})(z_2^{\nu_3}\overline{z_2}^{\nu_4})$, where $\nu_1,\nu_2\geq 0$ and $\nu_3+\nu_4\geq -\ell+k$.
The first factor is in $W^{k}(\D)$ (indeed for any $k$)
and the second factor is in $W^{k}(\D^*,\ell\phi)$.  It follows that the algebraic tensor product $W^{k}(\D)\tensor 
W^{k}(\D^*,\ell\phi)$  is  dense in $\pss^k(\p,\ell\Phi)$ and the result follows. 
\end{proof}

We can now complete the proof of Theorem~\ref{thm-main}. Since $u_\ell= K^\ell_{\h} (g)$,  it is sufficient to
show that the operator $K^\ell_\h$  is bounded from  $W^{2k}(\h,\ell\Phi)$ to $W^k(\h,(l+2k)\Phi)$.
We recall 
that   the representation of $K^\ell_\h$,  given by \eqref{eq-rep-klh},   is:
\[K^\ell_{\h} =  \mathsf{F}_0^*\circ K^{\ell+1}_\p\circ \mathsf{G}_1^*=\mathsf{F}_0^*\circ \left( K_{\D}\csor I_{\D^*}+ P_{\D}\csor K_{\D^*}^{\ell+1}\right)\circ \mathsf{G}_1^*.\]

Thanks to Lemma~\ref{prop-fgmapping},
the operator $\mathsf{G}_1^*$
which occurs as the first factor from the right,
is known to be continuous from $W^{2k}_{0,1}(\h,\ell\Phi)$ to $W^{2k}_{0,1}(\p, (\ell+1)\Phi)$.  Thanks to \eqref{eq-pssinclusions}, 
it now follows that the operator $\mathsf{G}_1^*$ is continuous from $W^{2k}_{0,1}(\h,\ell\Phi)$  to $\pss^k_{0,1}(\p, (\ell+1)\Phi)$.

 We claim that the canonical solution operator $ K^{\ell+1}_\p$, which occurs as the middle factor of the expression for $K^\ell_{\h}$
 is bounded from $\pss^k_{0,1}(\p, (\ell+1)\Phi)$ to $\pss^k(\p,(\ell+k+1)\Phi)$.
 Using  Lemma~\ref{lem-psstensor}
and the fact that forms of different degree are orthogonal by definition,  we have
\begin{equation}\label{eq-psskdecomp}
\pss^k_{0,1}(\p, (\ell+1)\Phi)= W^{k}_{0,1}(\D)\csor W^{k}(\D^*,(\ell+1)\phi)\oplus W^{k}(\D)\csor W^{k}_{0,1}(\D^*,(\ell+1)\phi),
\end{equation}
where $\oplus$ represents orthogonal direct sum of subspaces.
Now we look at the two terms in the expression for $K^{\ell+1}_\p$ which is the middle factor of \eqref{eq-rep-klh}:
\[
 K_{\D}\csor I_{\D^*}+ P_{\D}\csor K_{\D^*}^{\ell+1}.\]

In the first term, $K_{\D}$ is the canonical solution on the disc and maps $W^{k}_{0,1}(\D)$ 
to $W^{k+1}(\D)$. Since the inclusion $W^{k+1}(\D)\subset
W^{k}(\D)$ is continuous, it follows that $K_\D$ is continuous from  $W^{k}_{0,1}(\D)$ to  $W^{k}(\D)$.  Since the inclusion
$W^{k}(\D^*,(\ell+1)\phi)\subset W^{k}(\D^*,(\ell+k+1)\phi)$ is continuous, it follows that the identity map also is continuous from 
$W^{k}(\D^*,(\ell+1)\phi)$ to $W^{k}(\D^*,(\ell+k+1)\phi)$.  Moreover, we defined the canonical solution operator to be zero 
on functions. It follows now  from Lemma~\ref{lem-psstensor} that the operator $ K_{\D}\csor I_{\D^*}$  maps the space
$\pss^k(\p, (\ell+1)\Phi)$ continuously  into $\pss^k(\p,(\ell+k+1)\Phi)$

In the second term $P_{\D}$ is the harmonic projection on the disc, which vanishes
on $(0,1)$-forms, and consequently, this term acts only on the second summand in the orthogonal decomposition \eqref{eq-psskdecomp}
of $\pss^k_{0,1}(\p, (\ell+1)\Phi)$. For functions, the Bergman projection $P_\D$ on the disc preserves the space $W^{k}(\D)$. In 
Proposition~\ref{prop-dstarreg} 
we saw that $K^{\ell+1}_{\D^*}$ maps $W^{k}_{0,1}(\D^*, (\ell+1)\phi)$ continuously into $W^k(\D^*,(\ell+k+1)\phi)$. It follows from
 Lemma~\ref{lem-psstensor} that
$P_\D\csor K_{\D^*}^{\ell+1}$ also maps the space $\pss^k(\p, (\ell+1)\Phi)$ continuously  into $\pss^k(\p,(\ell+k+1)\Phi)$,
and the same is true of the sum $K^{\ell+1}_\p=K_{\D}\csor I_{\D^*}+ P_{\D}\csor K_{\D^*}^{\ell+1}$, which occurs as the middle factor of 
\eqref{eq-rep-klh}.


The continuous inclusions of \eqref{eq-pssinclusions} now imply that $K^{\ell+1}_\p$ maps  $W^{2k}_{0,1}(\p,(\ell+1)\Phi)$ continuously
into $W^k(\p,(\ell+k+1)\Phi)$.  But the factor $\mathsf{F}_0^*$ in \eqref{eq-rep-klh}, by Lemma~ \ref{prop-fgmapping},
 maps $W^k(\p,(\ell+k+1)\Phi)$ continuously into $W^k(\h,(\ell+2k)\Phi)$.
It follows that $K^\ell_{\h}$ is continuous from $W^{2k}_{0,1}(\h,\ell\Phi)$ to $W^k(\h,(\ell+2k)\Phi)$,  and Theorem~\ref{thm-main} is proved.

\section{Proof of Theorem~\ref{thm-2} }
By a    result of Kohn (see \cite[p.~82]{chen-shaw}), the Bergman projection $B$  on $\h$
can be represented in terms of the $\dbar$-Neumann operator
$\mathsf{N}$ as:
\begin{align*} B&= I- \dbar^* \mathsf{N} \dbar\\
&= I- K \dbar,
\end{align*}
 where $K= K^0_{\h}$ is the (unweighted) canonical solution operator on the domain $\h$. Now the $\dbar$ operator maps
$W^{2k+1}(\h)$ continuously into $W^{2k}_{0,1}(\h)$, and thanks to the regularity result for $K$ established in Theorem~\ref{thm-main},
it follows that $K$ maps $W^{2k}_{0,1}(\h)$ continuously into $W^{k}(\h, 2k\Phi)$.  Since the space 
$W^{k}(\h, 2k\Phi)$ continuously includes the space  $W^{2k+1}(\h)$, it follows by Kohn's formula above,
$B$ is continuous as well between these spaces. This proves the first  statement.

For the second statement, we can either repeat the argument used in the proof of Corollary~\ref{cor-blowup}, or we can use Kohn\rq{}s formula and  
Corollary~\ref{cor-blowup} directly: if $f\in \mathcal{C}^\infty(\overline{\h})$, clearly
 $\dbar f \in \mathcal{C}^\infty_{0,1}(\overline{\h})\cap \ker(\dbar)$, so by Corollary~\ref{cor-blowup}  we have 
$K\dbar f \in \mathcal{C}^\infty_{0,1}(\overline{\h}\setminus\{0\})$. Therefore, $Bf = f- K\dbar f$ is  in
 $\mathcal{C}^\infty_{0,1}(\overline{\h}\setminus\{0\})\cap\mathcal{O}(\h)$.

We claim that to show that $B$ does not map the space $\mathcal C^\infty_0(\h)$  into 
 $W^{1}(\h)$, it suffices to show that $W^{1}(\h)\cap\mathcal{O}(\h)$ is not dense in the Bergman space 
$L^2(\h)\cap\mathcal{O}(\h)$ in the $L^2$-topology. Indeed, if $\{f_n\}$ is a sequence of functions
in $\mathcal{C}^\infty_0(\h)$ which converge in $L^2$ to a Bergman function $f\in L^2(\h)\cap\mathcal{O}(\h)$,
then $Bf_n$ converges to $f$ in $L^2$. If $Bf_n\in W^{1}(\h)$, this would imply that $W^{1}(\h)$ is dense in $L^2(\h)$. 

To show that $W^{1}(\h)\cap\mathcal{O}(\h)$ is not dense in the Bergman space $L^2(\h)\cap\mathcal{O}(\h)$,
it is sufficient to find a non-zero  function  $f\in L^2(\h)\cap\mathcal{O}(\h) $ which lies in the orthogonal
complement of $W^{1}(\h)\cap\mathcal{O}(\h)$. We can take $f(w)= \displaystyle{\frac{1}{w_2}}$. This $f$  is in 
the Bergman space  $L^2(\h)\cap\mathcal{O}(\h)$, since using the standard biholomorphism from $\h$ to $\p$ 
given by $(w_1,w_2)\mapsto \left( \frac{w_1}{w_2},w_2\right)$ we obtain
\begin{align*} \int_{\h} \abs{\frac{1}{w_2}}^2 dV(w)
&= \int_\p \frac{1}{\abs{z_2}^2} \abs{z_2}^2 dV(z)\\
&= \pi^2.
\end{align*}
However, since $\frac{\partial f}{\partial w_2} = \frac{1}{w_2^2}$ is not square integrable on $\h$, it follows that 
$f$ is not in $W^{1}(\h)$.

Note that any holomorphic function on the domain $\p$ has a Laurent expansion
\[ \sum_{k=-\infty}^\infty \sum_{j=0}^\infty a_{j,k}z_1^jz_2^k,\]
which converges uniformly on compact subsets of $\p$.  Using again the biholomorphism 
$(w_1,w_2)\mapsto \left( \frac{w_1}{w_2},w_2\right)$, we see that every function in $\mathcal{O}(\h)$ has a Laurent 
expansion
\[ \sum_{k=-\infty}^\infty \sum_{j=0}^\infty a_{j,k}\left(\frac{w_1}{w_2}\right)^jw_2^k,\]
converging uniformly on compact subsets. However, if $j\geq 0$,$k\geq -1$, it is easily seen that each 
Laurent  monomial  $ \left(\frac{w_1}{w_2}\right)^jw_2^k$ is in $L^2(\h)$ and these monomials are
orthogonal. It easily follows from the convergence of the Laurent expansion that these monomials 
are complete in the Bergman space $L^2(\h)\cap\mathcal{O}(\h)$, i.e., their span is dense 
in the Bergman space. 

Now let $g\in W^{1}(\h)\cap \mathcal{O}(\h)$.  Since $ \frac{\partial g}{\partial z_2}\in L^2(\h)$
it follows that in the Laurent expansion of $g$, the coefficient of $\frac{1}{w_2}$ must be 0,  since otherwise,
the expansion of $\frac{\partial g}{\partial z_2}$ will have a term in $\frac{1}{w_2^2}$ which is not in $ L^2(\h)$.
Since the Laurent monomials are orthogonal in $L^2(\h)$,
it follows that $g$ is orthogonal to $f$ (which is a Laurent monomial $\frac{1}{w_2}$),  and our result is proved.

\bigskip
\noindent
{\bf Remarks:}  For a bounded  pseudoconvex domain $\Omega$  in $\cx^n$ with smooth boundary, the space 
$\mathcal C^\infty(\overline\Omega)\cap \mathcal O(\Omega)$ is dense in $L^2(\Omega)\cap \mathcal O(\Omega)$.
This follows from results due to Kohn (see \cite{Kohn}) on the  regularity of the weighted $\dbar$-Neumann  operator  $\mathsf{N}_t$,
 where  the weight 
function $t|z|^2$ with large $t>0$  is used (See  \cite[Theorem~8.1]{Sibony} for a detailed discussion.) 
Using the proof in \cite{Sibony} and the fact $W^1(\h)\cap\mathcal{O}(\h)$ is not dense in the Bergman space 
$L^2(\h)\cap\mathcal{O}(\h)$,  we see that the weighted Bergman projection $B_t$  on the Hartogs triangle is also not bounded from
  $\mathcal C^\infty_0(\h)$  to 
 $W^1(\h)$. The weights $t|z|^2$ can be substituted by any  functions  smooth up to the boundary.

We also mention that  using a result of Barrett (see \cite{Barrett2}),     the Bergman projection on each smooth Diederich-Fornaess worm domain
 $\Omega$  is not regular from $W^s$ to $W^s$ for some $s>0.$ But it is still an open question whether   on each worm  $B(C_0^\infty(\Omega))$
 is not contained in  $W^s(\Omega)$. Our example $\h$ is not smooth.   On the other hand, such examples exist  for  pseudoconvex domains with 
smooth boundary in complex manifolds (see \cite{Barrett1}).

\end{document}